\documentclass[a4paper,10pt]{article}

\usepackage{amsthm}
\usepackage{graphicx}
\usepackage{forest}

\usepackage{amssymb}
\usepackage[a4paper]{geometry}
\usepackage[parfill]{parskip} 
\usepackage{color}
\usepackage{url}
\usepackage{graphicx}
\usepackage{url}
\usepackage{float}

\usepackage{tikz}
\usepackage{subcaption}
\usepackage{ifthen}
\usepackage{intcalc}

\usetikzlibrary{arrows}

\newtheorem{theorem}{Theorem}
\newtheorem{lemma}[theorem]{Lemma}
\newtheorem{definition}[theorem]{Definition}
\newtheorem{observation}[theorem]{Observation}
\newtheorem{corollary}[theorem]{Corollary}

\newcommand{\cha}[2]{#1_{ #2}}
\newcommand{\inv}[2]{#1^{\not\sim #2}}

\def\odd{{\scriptstyle{\mathrm{odd}}}}
\def\even{{\scriptstyle{\mathrm{even}}}}

\begin{document}

\title{Not all phylogenetic networks are
leaf-reconstructible\thanks{LvI and MJ were funded in part by the Netherlands
Organization for Scientific Research (NWO), including Vidi grant
639.072.602 and LvI also partly by the 4TU Applied Mathematics Institute. PLE was supported in part by the National Research, Development and Innovation Office --- NKFIH grant K~116769 and KH~126853.}}

\author{P\'eter L. Erd\H{o}s\footnote{Alfr\'ed R{\'e}nyi Institute of Mathematics, Re\'altanoda u 13--15 Budapest, 1053 Hungary,
              erdos.peter@renyi.mta.hu} \and Leo van Iersel\footnote{Delft Institute of Applied Mathematics, Delft University of Technology,
              Van Mourik Broekmanweg 6,
2628XE Delft, The Netherlands,
             \{L.J.J.vanIersel, M.E.L.Jones\}@tudelft.nl} \and Mark Jones\footnotemark[3]}

\maketitle
\begin{abstract}
Unrooted phylogenetic networks are graphs used to represent
evolutionary relationships. Accurately reconstructing such networks is
of great relevance for evolutionary biology. It has recently been
conjectured that all phylogenetic networks with at least five leaves
can be uniquely reconstructed from their subnetworks obtained by
deleting a single leaf and suppressing degree-2 vertices. Here, we
show that this conjecture is false, by presenting a counter example
for each possible number of leaves that is at least~4. Moreover, we
show that the conjecture is still false when restricted to binary networks.
\end{abstract}

{\small {\bf Keywords:} Reconstruction; Phylogenetics }

{\small {\bf AMS Subject Classificiation:} 05C60  Isomorphism problems}

\section{Introduction}

The reconstruction conjecture, introduced in 1941 by Kelly and Ulam
(see~\cite{bh77}), conjectures that each graph with at least three
vertices is uniquely reconstructable from its multiset of
vertex-deleted subgraphs. Despite more than seven decades of research,
the conjecture is still open.

Recently, a variant of this conjecture was introduced that is relevant
for the field of \emph{phylogenetics}, the study of evolutionary
relationships. Such relationships among a set~$X$ of entities (e.g.
biological species or languages) are traditionally described by a tree
with no degree-2 vertices and its leaves bijectively labelled by the
elements of~$X$; this is called a \emph{phylogenetic tree} on~$X$.
More recently, evolutionary histories are more and more often
described by
phylogenetic \emph{networks}~\cite{expanding}, which are basically
(directed or undirected) graphs with their leaves bijectively labelled
by the elements of~$X$. These networks are able to describe more
complex evolutionary relationships than trees.

To find out whether it may be possible to accurately reconstruct phylogenetic networks, an important question to answer is which substructures uniquely define a phylogenetic network. For example, although there is quite some research directed at reconstructing rooted phylogenetic networks from embedded trees (see e.g.~\cite{kernel,maf}), these trees do not uniquely define a network (see e.g.~\cite{distinguish}). Hence, no method based on embedded
trees can guarantee to reconstruct the right network, even when it gets error-free and complete trees as input. Moreover, it has recently been shown that rooted phylogenetic networks also cannot be reconstructed uniquely from their subnetworks obtained by deleting one or more leaves and transforming the result into a valid rooted phylogenetic network~\cite{HvIMW2015}. A similar reconstruction question for pedigrees has also been answered negatively~\cite{thatte06}.

Here, we focus on \emph{unrooted phylogenetic networks}, which are
(undirected) graphs for which holds that contracting each
2-edge-connected component into a single vertex gives a phylogenetic
tree on~$X$. A recent paper~\cite{vIM2017} studied reconstructing such
networks from their $X$-deck, which consists of the graphs obtained by
deleting a single leaf from the network and suppressing its former
neighbour if its degree is~2, see Figure~\ref{fig:intro} for an
example. Several promising results were obtained, including a proof that all phylogenetic trees and all decomposable networks (i.e. networks with some cut-edge that is not incident to leaf) are reconstructable from their $X$-deck, as well as all networks with $|E|-|V|\leq 3$ and all networks with sufficiently many leaves (relative to $|E|-|V|$), all assuming that~$|X|\geq 5$. It was conjectured that all unrooted phylogenetic networks on~$X$, with~$|X|\geq 5$, can be uniquely reconstructed from their $X$-deck.

\begin{figure}
   \tikzset{lijn/.style={very thick}}
\centering\begin{tikzpicture}[ very thick, >= triangle 60]

\begin{scope}[xshift=-.5cm,yshift=-1cm]
\draw[very thick, fill, radius=0.09] (.5,0) circle node[left] {$x_1$};
\draw[very thick, fill, radius=0.09] (.5,-1) circle node[left] {$x_2$};
\draw[very thick, fill, radius=0.09] (2.5,1) circle node[above] {$x_3$};
\draw[very thick, fill, radius=0.09] (2.5,-2) circle node[below] {$x_4$};
\draw[very thick, fill, radius=0.09] (1,-.5) circle;
\draw[very thick, fill, radius=0.09] (1.5,-.5) circle;
\draw[very thick, fill, radius=0.09] (2,0) circle;
\draw[very thick, fill, radius=0.09] (2,-1) circle;
\draw[very thick, fill, radius=0.09] (2.5,-1.5) circle;
\draw[very thick, fill, radius=0.09] (2.5,.5) circle;
\draw[lijn] (.5,0) -- (1,-.5);
\draw[lijn] (.5,-1) -- (1,-.5);
\draw[lijn] (1.5,-.5) -- (1,-.5);
\draw[lijn] (1.5,-.5) -- (2.5,.5);
\draw[lijn] (1.5,-.5) -- (2.5,-1.5);
\draw[lijn] (2.5,.5) -- (2.5,-1.5);
\draw[lijn] (2.5,.5) -- (2.5,1);
\draw[lijn] (2.5,-2) -- (2.5,-1.5);
\draw[lijn] (2,-1) -- (2,0);
\draw (1.5,-3) node {$N$};
 \end{scope}
 \begin{scope}[xshift=4.5cm,yshift=0cm]
\draw[very thick, fill, radius=0.09] (2.5,1) circle node[above] {$x_3$};
\draw[very thick, fill, radius=0.09] (2.5,-2) circle node[below] {$x_4$};
\draw[very thick, fill, radius=0.09] (1,-.5) circle node[left] {$x_2$};
\draw[very thick, fill, radius=0.09] (1.5,-.5) circle;
\draw[very thick, fill, radius=0.09] (2,0) circle;
\draw[very thick, fill, radius=0.09] (2,-1) circle;
\draw[very thick, fill, radius=0.09] (2.5,-1.5) circle;
\draw[very thick, fill, radius=0.09] (2.5,.5) circle;
\draw[lijn] (1.5,-.5) -- (1,-.5);
\draw[lijn] (1.5,-.5) -- (2.5,.5);
\draw[lijn] (1.5,-.5) -- (2.5,-1.5);
\draw[lijn] (2.5,.5) -- (2.5,-1.5);
\draw[lijn] (2.5,.5) -- (2.5,1);
\draw[lijn] (2.5,-2) -- (2.5,-1.5);
\draw[lijn] (2,-1) -- (2,0);
 \end{scope}
 \begin{scope}[xshift=8.5cm,yshift=0cm]
\draw[very thick, fill, radius=0.09] (2.5,1) circle node[above] {$x_3$};
\draw[very thick, fill, radius=0.09] (2.5,-2) circle node[below] {$x_4$};
\draw[very thick, fill, radius=0.09] (1,-.5) circle node[left] {$x_1$};
\draw[very thick, fill, radius=0.09] (1.5,-.5) circle;
\draw[very thick, fill, radius=0.09] (2,0) circle;
\draw[very thick, fill, radius=0.09] (2,-1) circle;
\draw[very thick, fill, radius=0.09] (2.5,-1.5) circle;
\draw[very thick, fill, radius=0.09] (2.5,.5) circle;
\draw[lijn] (1.5,-.5) -- (1,-.5);
\draw[lijn] (1.5,-.5) -- (2.5,.5);
\draw[lijn] (1.5,-.5) -- (2.5,-1.5);
\draw[lijn] (2.5,.5) -- (2.5,-1.5);
\draw[lijn] (2.5,.5) -- (2.5,1);
\draw[lijn] (2.5,-2) -- (2.5,-1.5);
\draw[lijn] (2,-1) -- (2,0);
 \end{scope}
 \begin{scope}[xshift=4.5cm,yshift=-3cm]
\draw[very thick, fill, radius=0.09] (.5,0) circle node[left] {$x_1$};
\draw[very thick, fill, radius=0.09] (.5,-1) circle node[left] {$x_2$};
\draw[very thick, fill, radius=0.09] (3,-.5) circle node[right] {$x_4$};
\draw[very thick, fill, radius=0.09] (1,-.5) circle;
\draw[very thick, fill, radius=0.09] (1.5,-.5) circle;
\draw[very thick, fill, radius=0.09] (2,0) circle;
\draw[very thick, fill, radius=0.09] (2,-1) circle;
\draw[very thick, fill, radius=0.09] (2.5,-.5) circle;
\draw[lijn] (.5,0) -- (1,-.5);
\draw[lijn] (.5,-1) -- (1,-.5);
\draw[lijn] (1.5,-.5) -- (1,-.5);
\draw[lijn] (1.5,-.5) -- (2,0);
\draw[lijn] (1.5,-.5) -- (2,-1);
\draw[lijn] (2,-1) -- (2.5,-.5);
\draw[lijn] (2,0) -- (2.5,-.5);
\draw[lijn] (3,-.5) -- (2.5,-.5);
\draw[lijn] (2,-1) -- (2,0);
\draw (3.5,-2) node {$X$-deck of $N$};
 \end{scope}
\begin{scope}[xshift=8.5cm,yshift=-3cm]
\draw[very thick, fill, radius=0.09] (.5,0) circle node[left] {$x_1$};
\draw[very thick, fill, radius=0.09] (.5,-1) circle node[left] {$x_2$};
\draw[very thick, fill, radius=0.09] (3,-.5) circle node[right] {$x_3$};
\draw[very thick, fill, radius=0.09] (1,-.5) circle;
\draw[very thick, fill, radius=0.09] (1.5,-.5) circle;
\draw[very thick, fill, radius=0.09] (2,0) circle;
\draw[very thick, fill, radius=0.09] (2,-1) circle;
\draw[very thick, fill, radius=0.09] (2.5,-.5) circle;
\draw[lijn] (.5,0) -- (1,-.5);
\draw[lijn] (.5,-1) -- (1,-.5);
\draw[lijn] (1.5,-.5) -- (1,-.5);
\draw[lijn] (1.5,-.5) -- (2,0);
\draw[lijn] (1.5,-.5) -- (2,-1);
\draw[lijn] (2,-1) -- (2.5,-.5);
\draw[lijn] (2,0) -- (2.5,-.5);
\draw[lijn] (3,-.5) -- (2.5,-.5);
\draw[lijn] (2,-1) -- (2,0);
 \end{scope}
\end{tikzpicture}
\caption{Example of a phylogenetic network~$N$ and its $X$-deck.} \label{fig:intro}
\end{figure}
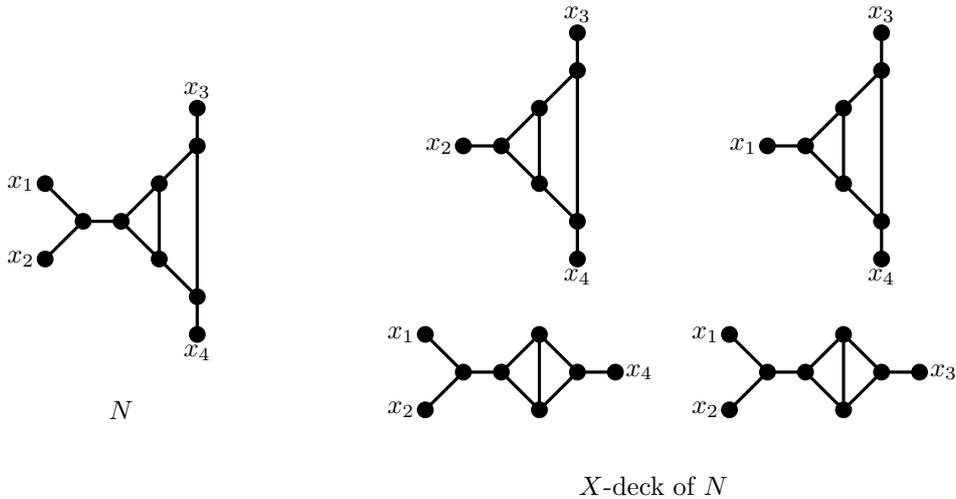

Here, we show that this conjecture is false. To do so, we present, for
each finite set~$X$ containing at least four elements, two unrooted
phylogenetic networks on~$X$ that are not isomorphic but have the same
$X$-deck. Moreover, we also give binary networks with these properties, hence showing that the conjecture restricted to binary networks is still false.

This result may have consequences for developing ``supernetwork'' methods, which attempt to reconstruct phylogenetic networks from subnetworks. Supertree methods work well for phylogenetic trees, which can be explained from the fact that a phylogenetic tree is uniquely determined by its induced set of four-leaved trees (or three-leaved trees in the case of rooted trees). Since phylogenetic networks are not uniquely determined by their subnetworks, developing supernetwork methods will be significantly more challenging than in the tree-case.

The structure of the paper is as follows. We start off by giving formal definitions related to phylogenetic networks and binary sequences, which are central to the construction of our counter examples. Then, in Section~\ref{sec:nonBinary}, we present our counter examples for the general, non-binary case. Finally, in Section~\ref{sec:binary} we show how these can be transformed into counter examples for the binary case.

\section{Preliminaries}\label{sec:prelim}
A \emph{phylogenetic tree} on $X$ is an undirected simple tree, with no degree-$2$ vertices, such that each leaf is bijectively labelled by an element from $X$.  A \emph{biconnected component} of a graph is a maximal $2$-edge-connected subgraph and it is called a \emph{blob} if it contains at least two edges.  Let $X$ be a finite set with $|X| \geq 2$, and let $N$ be an undirected simple graph in which the \emph{leaves} (degree-$1$ vertices) are bijectively labelled by the elements of $X$.  We say $N$ is an \emph{unrooted phylogenetic network} on $X$  if contracting each blob into a single vertex gives phylogenetic tree (or equivalently, each cut-edge induces a unique partition of the leaves).  In addition, we say that $N$ is \emph{binary} if every vertex has degree $1$ or $3$.  In what follows, we will refer to unrooted phylogenetic networks as \emph{networks} for short.

Let $G$ and $H$ be two partially labelled undirected multigraphs with the same label set, such that $|V(G)| = |V(H)|$. Let $f: V(G) \rightarrow V(H)$ be a bijective function. We say that $f$ is an \emph{isomorphism} between $G$ and $H$ if it is both label-preserving (that is, vertex $a \in V(U)$ has label $l$ if and only if $f(a)$ has label $l$) and edge-preserving (that is, for any $a,b \in V(G)$ the number of edges between $a$ and $b$ in $G$ is equal to the number of edges between $f(a)$ and $f(b)$ in $H$). We say $G$ and $H$ are \emph{equivalent}, denoted $G \sim H$, if there is an isomorphism between $G$ and $H$.

Given an undirected multigraph $G$ with no vertices of degree $2$, and a vertex $a \in V(G)$, we denote by $G_a$ the undirected multigraph derived from $G$ by deleting $a$ and all incident edges, and then suppressing any degree-$2$ vertices. We say $G_a$ is derived from $G$ by \emph{removing} the vertex $a$. For a label $x$, we may write $G_{x}$ to refer to $G_{a}$, where $a$ is the unique vertex in $G$ with label $x$.

Given a network $N$ on $X$, an $X$-reconstruction of $N$ is a network $N'$ on $X$ such that $N'_x \sim N_x$ for all $x \in X$. We call a phylogenetic network $N$ \emph{leaf-reconstructible} if $N' \sim N$ for  every $X$-reconstruction $N'$ of $N$.  That is, all $X$-reconstructions of $N$ are isomorphic to each other.

It was conjectured in~\cite{vIM2017} that all unrooted phylogenetic networks with $5$ or more leaves are leaf-reconstructible. (We note that phylogenetic trees on $5$ or more leaves are leaf-reconstructible, as it is clearly possible to reconstruct every quartet in the tree.)

In this paper, we show that the conjecture is false.  More precisely, we will show that for each $r \geq 4$, there exist binary unrooted phylogenetic networks $N$ and $N'$ on $X$ with $|X|=r$, such that $N \not\sim N'$, but $N_x \sim N'_x$ for all $x \in X$. Thus, $N$ and $N'$ are not leaf-reconstructible. \footnote{It was previously known that networks on $r=4$ leaves are not leaf-reconstructible in general. We nevertheless include the case $r = 4$ in our paper, as it allows us to  give simpler figures than for the $r = 5$ case.}

Finally, for an integer $k$, let $[k]$ denote the set $\{1,2,\dots, k\}$.

\subsection{Binary sequences}\label{subsec:sequences}
Given an alphabet $\Sigma$, let $w \in \Sigma^*$ be a sequence of elements with elements drawn from $\Sigma$. If $\Sigma = \{0,1\}$ then we call $w$ a \emph{binary sequence}.  The \emph{length} of the sequence $w$, denoted $l(w)$, is the number of elements in $w$. We write $\cha{w}{i}$ to denote the $i$'th element of $w$.  We often write $e_1e_2\dots e_l$ to denote the sequence $w$ such that $l(w) = l$ and $\cha{w}{i} = e_i$ for each $i \in [l]$.  (Thus, for example, $1011$ denotes the length-$4$ binary sequence whose second element is $0$ and whose first, third and fourth elements are $1$.)  Given a binary sequence $w$, the \emph{weight} of $w$ is the number of $1$'s in $w$.
For an integer $l$, we write ${\cal B}_l$ to denote the set of binary sequences of length $l$. Given a sequence $w \in {\cal B}_r$  and $i \in [r]$, let $\inv{w}{i}$ be the sequence derived from $w$ by replacing the $i$'th element with $1 - \cha{w}{i}$ (for example, if $w = 1001$ and $i = 3$, then $\inv{w}{i} = 1011$).

Central to the proof of our result is the idea that for a binary sequence $w$, one needs to know all elements of $w$ in order to decide  whether $w$ has odd or even weight. (Note that here and in the rest of the paper, we consider a sequence of weight $0$  to have even weight.) For some integer $r$, consider the set $S^{\even}$ of all length-$r$ binary words of even weight, and the set $S^{\odd}$ of all length-$r$ binary words of odd weight.  Given a length-$r$ binary sequence $w$ and integer $i \in [r]$, let $w^{-i}$ denote the sequence on $\{0,1,*\}$ derived from $w$ by replacing the $i$'th element with $*$.
Then for each $w \in S^{\even}$, there exists a sequence $w' \in S^{\odd}$ such that $(w')^{-i} = w^{-i}$  (indeed, $\inv{w}{i}$ is such a sequence).
For a set of sequences $S$ and $i \in [r]$, let $S^{-i} = \{w^{-i}; w \in S\}$. Then it follows that for each $i \in [r]$, the sets $({S^\odd})^{-i}$ and $({S^\even})^{-i}$
are the same.

We will use this concept to guide our construction of two networks $N^{\even}$ and $N^{\odd}$ on a set $X = \{x_1, \dots, x_r\}$. Roughly speaking, $N^{\even}$ can be thought of as a representation of $S^{\even}$, and $N^{\odd}$ can be thought of as a representation of $S^{\odd}$. Then for each $i \in [r]$,  $(N^{\even})_{x_i}$ corresponds to $(S^{\even})^{-i}$, and $(N^{\odd})_{x_i}$ corresponds to $(S^{\odd})^{-i}$. Just as $(S^{\even})^{-i} = (S^{\odd})^{-i}$, we will be able to show that $(N^{\even})_{x_i}$ and $(N^{\odd})_{x_i}$ are equivalent, while originally  $N^{\even}$ and $N^{\odd}$ are different.

\section{Non-binary example}\label{sec:nonBinary}

In order to demonstrate the main concepts of our construction, we first give a construction using non-binary graphs.
In the next section, we will construct an example with binary phylogenetic networks, using these non-binary graphs as a guide.

For some integer $r \geq 4$, let $X$ denote the set of labels $\{x_1, \dots x_r\}$. We will construct two graphs $M^{\even}$ and $M^{\odd}$, in which the leaves are bijectively labelled by the elements of $X$. As in the previous section, let $S^{\even}$ denote the set of all length-$r$ binary words of even weight, and let $S^{\odd}$ denote the set of all length-$r$ binary words of odd weight.

The graph $M^{\even}$ is constructed as follows. For each $i \in [r]$, let $M^{\even}$ contain vertices $v_{i,0}$ and $v_{i,1}$, and a leaf labelled with $x_i$, such that $x_i$ is adjacent to $v_{i,0}$.\footnote{We note that in this section and next, we will often give names to particular vertices in the graphs we construct. This is done to differentiate between vertices, in order to aid in the description of the construction and help define isomorphisms. However, this is not the same as labelling the vertices; the only labelling that will occur is the labelling of leaves with elements of $X$.}
For each $w \in S^{\even}$, let $M^{\even}$ contain a vertex $u_w$. For each $w \in S^{\even}$ and $i \in [r]$, let $u_w$ be adjacent to $v_{i,0}$ if $\cha{w}{i} = 0$, and let $u_w$ be adjacent to $v_{i,1}$ if $\cha{w}{i} = 1$.
This completes the construction of $M^{\even}$. (See Figure~\ref{fig:NEvenNonBinary}.)

The construction of $M^{\odd}$ is identical to that of $M^{\even}$, except that we have a node $u_w$ for each $w \in S^{\odd}$ rather than each $w \in S^{\even}$. For completeness, the full construction is as follows:
For each $i \in [r]$, let $M^{\odd}$ contain vertices $v_{i,0}$ and $v_{i,1}$, and a leaf labelled with $x_i$, such that $x_i$ is adjacent to $v_{i,0}$.
For each $w \in S^{\odd}$, let $M^{\odd}$ contain a vertex $u_w$.
For each $w \in S^{\odd}$ and $i \in [r]$, let $u_w$ be adjacent to $v_{i,0}$ if $\cha{w}{i} = 0$, and let $u_w$ be adjacent to $v_{i,1}$ if $\cha{w}{i} = 1$. This completes the construction of $M^{\odd}$. (See Figure~\ref{fig:NOddNonBinary}.)

\begin{figure}[h]
\begin{subfigure}{\textwidth}
  \tikzstyle{black}=[circle, fill=black, minimum size=1mm, radius=0.09, inner sep = 2.3]
 \centering\begin{tikzpicture}[ very thick, >= triangle 60, scale = 0.8]
\foreach \eps in {0.45}
{
 \foreach \xi in {1,2,3,4}
 {
  \node at (4*\xi-3.5,1)(root\xi0)[black]{};
  \node at (4*\xi-3.5-1.5*\eps,1){$v_{\xi,0}$};
  \node at (4*\xi-2.5,1)(root\xi1)[black]{};
  \node at (4*\xi-2.5,1-\eps){$v_{\xi,1}$};
  \node at (4*\xi-3.5,0)[black]{}
   edge (root\xi0);
  \node at (4*\xi-3.5,0-\eps){$x_{\xi}$};
 }
 \foreach \ha in {0,1}
 \foreach \hb in {0,1}
 \foreach \hc in {0,1}
 \foreach \hd in {0,1}
 {
 \ifthenelse{\intcalcMod{\ha + \hb + \hc + \hd}{2} = 0}{
  \node at (8*\ha + 4*\hb + 2*\hc,6)(root)[black]{}
  edge (root1\ha)
  edge (root2\hb)
  edge (root3\hc)
  edge (root4\hd)
  ;
  \node at (8*\ha + 4*\hb + 2*\hc,6+\eps)(root){$u_{\ha \hb \hc \hd}$};
 }{}
 }
}
\end{tikzpicture}
\caption{$M^{\even}$}\label{fig:NEvenNonBinary}
\end{subfigure}

\begin{subfigure}{\textwidth}
\label{fig:hardnessGadgetForX}
  \tikzstyle{black}=[circle, fill=black, minimum size=4pt, radius=0.09, inner sep = 2.3]
  \tikzstyle{red}=[square, fill=red, minimum size=4pt]
  \tikzstyle{demand}=[blue, font = \Large]
  \tikzstyle{zeroDemand}=[blue]
\centering\begin{tikzpicture}[ very thick, >= triangle 60, scale = 0.8]
\foreach \eps in {0.45}
{
 \foreach \xi in {1,2,3,4}
 {
  \node at (4*\xi-3.5,1)(root\xi0)[black]{};
  \node at (4*\xi-3.5-1.5*\eps,1){$v_{\xi,0}$};
  \node at (4*\xi-2.5,1)(root\xi1)[black]{};
  \node at (4*\xi-2.5,1-\eps){$v_{\xi,1}$};
  \node at (4*\xi-3.5,0)[black]{}
   edge (root\xi0);
  \node at (4*\xi-3.5,0-\eps){$x_{\xi}$};
 }
 \foreach \ha in {0,1}
 \foreach \hb in {0,1}
 \foreach \hc in {0,1}
 \foreach \hd in {0,1}
 {
 \ifthenelse{\intcalcMod{\ha + \hb + \hc + \hd}{2} = 1}{
  \node at (8*\ha + 4*\hb + 2*\hc,6)(root)[black]{}
  edge (root1\ha)
  edge (root2\hb)
  edge (root3\hc)
  edge (root4\hd)
  ;
  \node at (8*\ha + 4*\hb + 2*\hc,6+\eps)(root){$u_{\ha \hb \hc \hd}$};
 }{}
 }
}
\end{tikzpicture}
\caption{$M^{\odd}$}\label{fig:OddNonBinary}
\end{subfigure}
 \caption{ Non-binary example of  $M^{\even}$ and $M^{\odd}$  for the case when when $r=4$. Nodes $u_w$ are adjacent to nodes $v_{i,h}$ if and only if $\cha{w}{i} = h$.}
\end{figure}
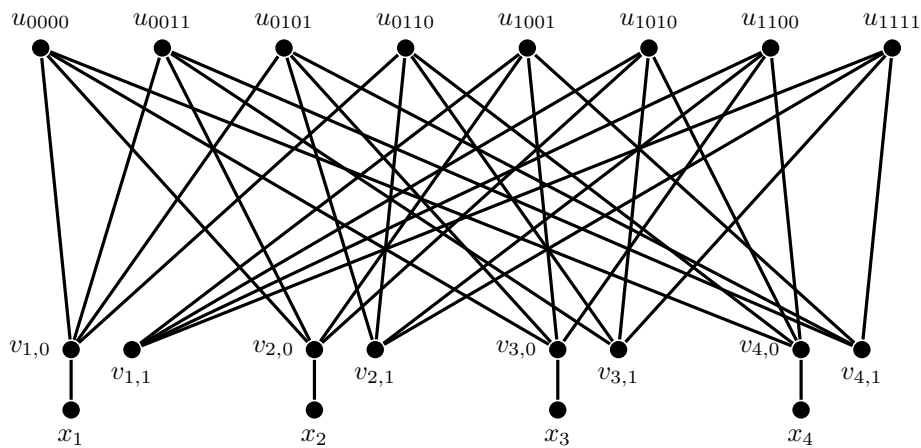
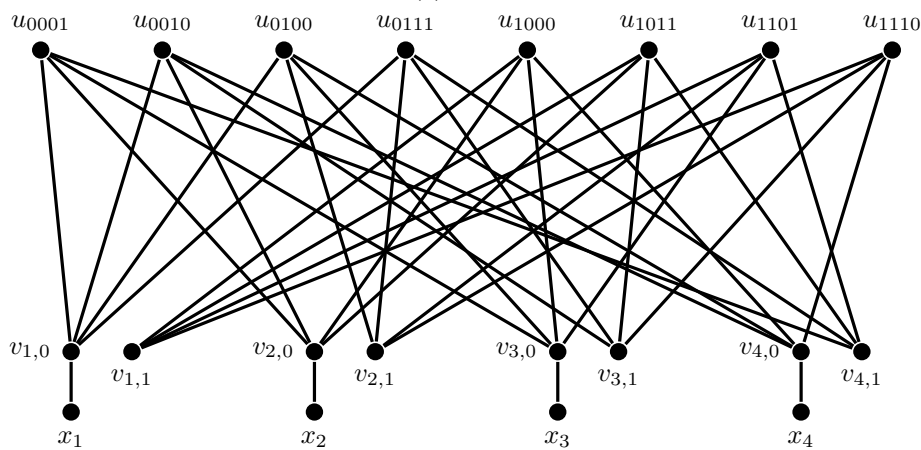

\bigskip

\begin{lemma}
 $M^{\even}$ and $M^{\odd}$ are not equivalent.
\end{lemma}
\begin{proof}
Suppose for a contradiction that $M^{\even}$ and $M^{\odd}$ are equivalent, and let $f:V(M^{\even}) \rightarrow V(M^{\odd})$ be an isomorphism between $M^{\even}$ and $M^{\odd}$. Let  ${\bf 0}$ denote the all-$0$ sequence from $S^{\even}$. Observe that for each $i \in [r]$, the distance between $u_{\bf 0}$ and $x_i$ is $2$ (as both $u_{\bf 0}$ and $x_i$ are adjacent to $v_{i,0}$). It follows that $f(u_{\bf 0})$ must have distance $2$ to $f(x_i) = x_i$ in  $M^{\odd}$ , for each $i \in [r]$. We will show that no such $f(u_{\bf 0})$ exists in $M^{\odd}$, a contradiction to the existence of $f$.

Observe that by construction of $M^{\odd}$ (in particular, the fact that it is a bipartite graph with one side consisting of vertices $v_{j,0}$ or $v_{j,1}$), the distance between any leaf $x_i$ and any vertex $v_{j,0}$ or $v_{j,1}$ is odd. It follows that $f(u_{\bf 0})$ must be the vertex $u_w$, for some $w \in S^{\odd}$ (any other vertex is either a leaf, which has distance $0$ from itself, or has odd distance from any leaf). However, for any $w \in S^{\odd}$ there exists $i \in [r]$ such that $\cha{w}{i} = 1$, and so $u_w$ is not adjacent to $v_{i,0}$.  As $v_{i,0}$ is the only vertex adjacent to $x_i$, it follows that the distance between $u_w$ and $x_i$ is greater than $2$, and so $f(u_{\bf 0}) \neq u_w$.

As there is no choice for  $f(u_{\bf 0})$ that satisfies the conditions of an isomorphism, we have that there is no possible isomorphism between $M^{\even}$ and $M^{\odd}$, and so $M^{\even}$ and $M^{\odd}$ are not equivalent.
\end{proof}

\begin{lemma}
For each $i \in [r]$,  $(M^{\even})_{x_i} \sim (M^{\odd})_{x_i}$.
\end{lemma}
\begin{proof}
Observe that $v_{i,0}$ and $v_{i,1}$ each have $2^{r-2} \geq 4$ neighbors in $M^{\even}$ not including $x_i$ (as $|S^{\even}| = 2^{r-1}$ and exactly half of the sequences in $S^{\even}$ have $1$ as their $i$'th element). Also any vertex $u_w$ has $r \geq 4$ neighbors in $M^{\even}$. It follows that if $x_i$ is deleted from $M^{\even}$, the remaining graph has no vertices of degree $2$, and thus  $(M^{\even})_{x_i}$ is exactly $M^{\even}$ with $x_i$ deleted.
By a similar argument, $(M^{\odd})_{x_i}$ is exactly $M^{\odd}$ with $x_i$ deleted.

Now define a bijective function $f: V((M^{\even})_{x_i}) \rightarrow V((M^{\odd})_{x_i})$ as follows. For each $w \in S^{\even}$, let $f(u_w) = u_{\inv{w}{i}}$. Observe that this defines a bijection between $\{u_w: w \in S^{\even}\}$ and  $\{u_w: w \in S^{\odd}\}$. Let $f(v_{i,0}) = v_{i,1}$ and $f(v_{i,1}) = v_{i,0}$. For $j \in [r]\setminus \{i\}$, let $f(v_{j,0}) = v_{j,0}, f(v_{j,1}) = v_{j,1}$ and $f(x_j) = x_j$ (recall that the leaf $x_i$ does not appear in $(M^{\even})_{x_i}$ or $(M^{\odd})_{x_i}$, so we do not need to define $f(x_i)$).

By construction, $f$ is a bijective function from $V((M^{\even})_{x_i})$ to $V((M^{\odd})_{x_i})$. It remains to show that $f$ is label-preserving and edge-preserving. As $f$ is the identity on all labelled vertices, $f$ is label-preserving. As $(M^{\even})_{x_i}$ and $(M^{\odd})_{x_i}$ are simple graphs, to show that $f$ is edge-preserving it is enough to show that two vertices $a,b$ are adjacent in $(M^{\even})_{x_i}$ if and only if $f(a)$ and $f(b)$ are adjacent in $(M^{\odd})_{x_i}$.

So consider any $a,b \in V((M^{\even})_{x_i})$.
Suppose first that $a = u_w$ for some $w \in S^{\even}$ and that $b = v_{j,h}$ for some $j \in [r]\setminus\{i\}$ and $h \in \{0,1\}$. Then $a$ and $b$ are adjacent if and only if $\cha{w}{j} = h$. Furthermore $f(a) = u_{\inv{w}{i}}$ where $\cha{\inv{w}{i}}{j} = \cha{w}{j}$, and $f(a)$ and $f(b) = v_{j,h}$ are adjacent if and only if $\cha{\inv{w}{i}}{j} = h$. Putting it together, we have that $ab \in E((M^{\even})_{x_i}) \Leftrightarrow \cha{w}{j} = h \Leftrightarrow \cha{\inv{w}{i}}{j}=h \Leftrightarrow f(a)f(b) \in E((M^{\even})_{x_i})$. Thus $a$ and $b$ are adjacent if and only $f(a)$ and $f(b)$  are adjacent.

Next suppose that $a = u_w$ for some $w \in S^{\even}$ and that $b = v_{i,h}$ for some and $h \in \{0,1\}$. Then $a$ and $b$ are adjacent if and only if $\cha{w}{i} = h$. Furthermore $f(a) = u_{\inv{w}{i}}$ where $\cha{\inv{w} {i}}{i} = 1 - \cha{w}{i}$, and $f(a)$ and $f(b) = v_{i,1-h}$ are adjacent if and only if $\cha{\inv{w}{i}}{i} = 1-h$. Thus $ab \in E((M^{\even})_{x_i}) \Leftrightarrow \cha{w}{j} = h \Leftrightarrow \cha{\inv{w}{i}}{j}= 1- h \Leftrightarrow f(a)f(b) \in E((M^{\even})_{x_i})$.

If $a$ and $b$ are $u_w, u_{w'}$ for some $w,w' \in S^{\even}$, then $a$ and $b$ are not adjacent, and neither are $f(a)$ and $f(b)$ (which are both vertices $u_{w''},u_{w'''}$ for some $w',w'' \in S^{\odd}$). By a similar argument, if $a$ and $b$ are both vertices $v_{j,h}$ for some $j \in [r]$ and $h \in \{0,1\}$, then $a,b$ are not adjacent and $f(a),f(b)$ are not adjacent. If $b = x_j$ for some $j \in [r] \setminus {j}$, then $a$ and $b$ are adjacent if and only if $a = v_{j,0}$, which holds if and only if $f(a) = v_{j,0}$, which in turn holds if and only if $f(a)$ is adjacent to $x_j = f(b)$. This covers all possible cases, and so we have that $a$ and $b$ are adjacent if and only if $f(a)$ and $f(b)$ are adjacent. This completes the proof that $f$ is an isomorphism, and so  $(M^{\even})_{x_i} \sim (M^{\odd})_{x_i}$.
\end{proof}

\section{Binary example}\label{sec:binary}
In this section, we show how to construct two binary networks on $X$ that are $X$-reconstructions of each other but are not equivalent, for $|X| \geq 4$.
This is enough to show that networks on $r \geq 4$  leaves are not leaf-reconstructible.

Given the non-binary networks $M^{\even}$ and $M^{\odd}$ constructed in the previous section, we proceed to construct two graphs $G^{\even}$ and $G^{\odd}$ in the following way. For each binary sequence $w \in {\cal B}_r$, $u_w$ will be expanded into a caterpillar $Cat(w)$ (details of the construction are given below). Each vertex $v_{i,h}$ will be expanded into a lexicographic tree $Lex({i,h})^\even$ or $Lex({i,h})^\odd$  (defined below).
For any $w \in {\cal B}_r, i \in [r], h \in \{0,1\}$, the subgraphs $Cat(w)$, $Lex({i,h})^\even$  and $Lex({i,h})^\odd$ contain leaves denoted $z_{w,i}$, for $w \in {\cal B}_r$ and $i \in [r]$. Two subgraphs $Cat(w)$ and $Lex({i,h})^\even$ (or $Cat(w)$ and $Lex({i,h})^\odd$ ) will share a vertex $z_{w,i}$ if and only if $\cha{w}{i}=h$ (analagous to how in $M^{\even}$ and $M^{\odd}$, the vertices $u_w$ and $v_{i,h}$ are adjacent if and only if $\cha{w}{i}=h$).

Similarly with $M^{\even}$ and $M^{\odd}$, we will  show that $G^{\even}$ and $G^{\odd}$ are not equivalent, but that they become equivalent if a single leaf $x_i$ is deleted.

We note that $G^{\even}$ and $G^{\odd}$ are not technically networks, because while they have maximum degree $3$, they contain some vertices of degree $2$ (in particular, every vertex  $z_{w,i}$ has degree $2$). In the last part of this section, we will produce two networks  $N^{\even}$ and $N^{\odd}$ from  $G^{\even}$ and $G^{\odd}$.

We now define the two types  of tree that will be used in our construction.

\medskip

\begin{definition}
For any sequence $w \in {\cal B}_r$, the \emph{caterpillar} $Cat(w)$ is the tree with internal vertices $u_w$ and $y_{w,i}$ for each $i \in [r-3]$, leaves $z_{w,i}$ for each $i \in [r]$, and edges
$u_wz_{w,1},u_wz_{w,2}, u_wy_{w,1}$, $y_{w,r-3}z_{w,r-1}, y_{w,r-3}z_{w,r}$,  and $y_{w,i}z_{w,i+2}, y_{w,i}y_{w,i+1}$ for each $1 \leq i \leq r-4$.
\end{definition}

See Figure~\ref{fig:caterpillar} for an example. Observe that all internal vertices of $Cat(w)$ have degree $3$.

\begin{figure}
\begin{center}
 \includegraphics[scale =1]{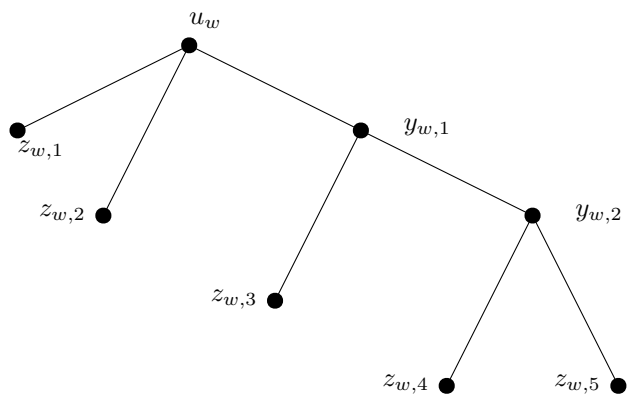}
\end{center}
 \caption{The caterpillar $Cat(w)$ for the case $r=5$.}
 \label{fig:caterpillar}
\end{figure}

\medskip

\begin{observation}\label{lem:caterpillarIsomorphism}
Given sequences $w,w' \in {\cal B}_r$, the trees $Cat(w)$ and $Cat({w'})$ are equivalent. In particular, there exists an isomorphism $f$ between $Cat(w)$ and $Cat({w'})$ such that $f(u_w)=u_{w'}$ and $f(z_{w,i}) = z_{w',i}$ for all $i \in [r]$.
\end{observation}

\medskip

\begin{definition}
Given a set $S$ of binary sequences such that $|S| = 2^t$ for some positive integer $t$, and and $i \in [r]$, the lexicographic tree $Lex(S,i)$ is a fully balanced binary tree with leaves $z_{w,i}$ for  $w \in S$. All non-leaf vertices have degree $3$ except for a single vertex, called the \emph{root}, of degree $2$, and all leaves are of distance exactly $t$ from the root.
Moreover, the leaves are arranged in such a way that there exists a depth-first search of the vertices of $Lex(S,i)$ that traverses the leaves $z_{w,i}$ in lexicographic order with respect to $w$. (Note that this uniquely determines $Lex(S,i)$.)
\end{definition}

\medskip

\begin{definition}
Let $(S^{\even})_{i:h}$ be the set of all length-$r$ binary sequences $w$ of even weight such that $\cha{w}{i} = h$. Let $(S^{\odd})_{i:h}$ be the set of all length-$r$ binary sequences $w$ of odd weight such that $\cha{w}{i} = h$.
\end{definition}

\medskip

\begin{definition}
For any $i \in [r]$ and $h \in \{0,1\}$,
define $Lex(i,h)^\even = Lex((S^{\even})_{i:h},i)$, and define $Lex(i,h)^\odd = Lex((S^{\odd})_{i:h},i)$. (Thus the leaves of $Lex(i,h)^\even$  are  $z_{w,i}$ for  $w \in S^\even_{i:h}$, and the leaves of $Lex(i,h)^\odd$  are  $z_{w,i}$ for  $w \in S^\odd_{i:h}$). We refer to the root of $Lex(i,h)^\even$ by $v_{i,h}^\even$, and we refer to the root of $Lex(i:h)^\odd$ by $v_{i,h}^\odd$.
\end{definition}

\medskip

(See Figure~\ref{fig:smallRootedLexicographicTrees} for some examples.)

\medskip

\begin{figure}\centering
\begin{subfigure}{0.3\textwidth}
\begin{center}
 \includegraphics[scale = 1]{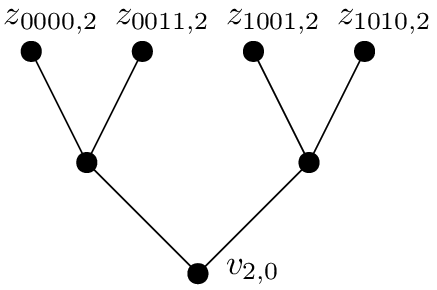}
\end{center}
 \caption{ $Lex({2,0})^{\even}$}    
\end{subfigure}
\begin{subfigure}{0.3\textwidth}
\begin{center}
 \includegraphics[scale = 1]{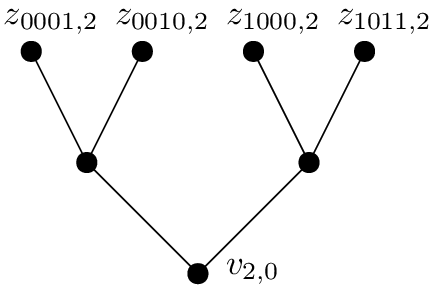}
\end{center}
 \caption{$Lex({2,0})^{\odd}$}    
\end{subfigure}
 \caption{
The lexicographic trees $Lex({2,0})^{\even}$ and $Lex({2,0})^{\odd}$ for the case $r = 4$. Leaves of $Lex({2,0})^{\even}$ (respectively, $Lex({2,0})^{\odd}$) are $z_{w,2}$ for every length-$r$ sequence $w$ of even weight (odd weight) such that $\cha{w}{2} = 0$.
 }\label{fig:smallRootedLexicographicTrees}
\end{figure}

\begin{lemma}\label{lem:specificLexTreeIsomorphism}
For any $j \in [r] \setminus \{i\}$ and $h \in \{0,1\}$, there exists an isomorphism $f$ between $Lex(j,h)^\even$ and $Lex(j,h)^\odd$ such that $f(v_{j,h}^\even) = v_{j,h}^\odd$, and $f(z_{w,j}) = z_{\inv{w}{i},j}$ for all $w \in (S^{\even})_{j:h}$.

Also, for any $h \in \{0,1\}$ there exists an isomorphism $f$ between
$Lex(i,h)^\even$ and $Lex(i,1-h)^\odd$ such that $f(v_{i,h}^\even) = v_{i,1-h}^\odd$, and $f(z_{w,i}) = z_{\inv{w}{i},i}$ for all $w \in (S^{\even})_{i:h}$.
\end{lemma}

\begin{proof}
Observe that the root of a lexicographic tree is unique, as it is the only vertex of degree $2$. Then for any integer $l$ and leaf $z_{w,j}$ in a lexicographic tree, we may define the \emph{depth-$l$ ancestor} of $z_{w,j}$ as follows. The depth-$l$ ancestor of $z_{w,j}$ is the unique vertex on a path between $z_{w,j}$ and the root, that has distance $l$ from $z_{w,ji}$ . Note that we count the root itself as a depth-$(r-2)$ ancestor of every leaf, and each leaf is the depth $0$ ancestor of itself. Moreover, because a lexicographic tree is fully balanced, if a vertex $a$ is the depth-$l$ ancestor of one leaf and the depth-$l'$ ancestor of  another leaf then $l = l'$.

In order to prove the first claim, we first show that for any two sequences $w,w' \in (S^{\even})_{j:h}$ and integer $l$, the leaves $z_{w,j}, z_{w',j}$ share a depth-$l$ ancestor in $Lex(j,h)^\even$  if and only if $z_{\inv{w}{i},j}$, $z_{\inv{w'}{i},j}$ share a depth-$l$ ancestor in $Lex(j,h)^\odd$. Indeed, it is easy to see that $z_{w,j}, z_{w',j}$  share a depth-$l$ ancestor if and only if $w,w'$ agree on the first $r-2-l$ elements not including $j$. But if $w,w'$ agree on these elements then so do $\inv{w}{i}, \inv{w'}{i}$, and so $z_{\inv{w}{i},j}$, $z_{\inv{w'}{i},j}$ also share a depth-$l$ ancestor.

Thus, we may define a bijective function $f:V(Lex(j,h)^\even) \rightarrow V(Lex(j,h)^\odd)$ as follows. For any vertex $a \in V(Lex(j,h)^\even)$ with distance $r-2-l$ from the root, choose any sequence $w \in (S^{\even})_{j:h}$ such that $a$ is a depth-$l$ ancestor of $z_{w,j}$, and let $f(a)$ be the depth-$l$ ancestor of $z_{\inv{w}{i},j}$ in $Lex(j,h)^\odd$. Observe that $f$ is well-defined, since we have just shown that if two leaves $z_{w,j}, z_{w',j}$ share $a$ as a depth-$l$ ancestor, then $z_{\inv{w}{i},j}$, $z_{\inv{w'}{i},j}$ also have the same depth-$l$ ancestor.

By construction, it is clear that  $f(v_{j,h}^\even) = v_{j,h}^\odd$, and $f(z_{w,j}) = z_{\inv{w}{i},j}$ for all $w \in (S^{\even})_{j:h}$. To see that $f$ is an isomorphism it remains to show that $f$ is edge-preserving. To see this, observe that two vertices $a,b \in V(Lex(j,h)^\even)$ are adjacent if and only if one is the depth-$l$ ancestor and the other the depth$(l+1)$ ancestor of some leaf, and that this holds if and only if $f(a),f(b)$ are also adjacent.

The proof of the second claim is similar.
\end{proof}

We can now describe the structure of $G^{\even}$ and $G^{\odd}$.

For each $w \in S^{\even}$, let $G^{\even}$ contain the caterpillar $Cat(w)$.
For each $i \in [r]$ and $h \in \{0,1\}$, let $G^{\even}$ contain the lexicographic tree $Lex({i,h})^{\even}$. Finally, for each $i \in [r]$ let $G^{\even}$ contain the labelled leaf $x_i$ adjacent to $v_{i,0}$.

The construction of $G^{\odd}$ is similar: For each $w \in S^{\odd}$, let $G^{\odd}$ contain the caterpillar $Cat(w)$. For each $i \in [r]$ and $h \in \{0,1\}$, let $G^{\odd}$ contain the lexicographic tree $Lex({i,h})^{\odd}$.
Finally, for each $i \in [r]$ let $G^{\odd}$ contain the labelled leaf $x_i$ adjacent to $v_{i,0}$.

Observe that in both $G^{\even}$ and $G^{\odd}$, the vertices $z_{w,i}$ have degree $2$ (as they appear as a leaf in the caterpillar $Cat(w)$  and in the lexicographic tree $Lex(i,\cha{w}{i})^\even$ or $Lex(i,\cha{w}{i})^\odd$ ).
The vertices $v_{i,1}$ also have degree $2$, and all other non-leaf vertices have degree $3$.

We will later show that $G^{\even}$ and $G^{\odd}$ are not equivalent. First though, we will show that the multigraphs derived from $G^{\even}$ and $G^{\odd}$ by \emph{deleting} (not removing) the same leaf are in fact equivalent.

\medskip

\begin{lemma}\label{lem:GiEquivalence}
For $i \in [r]$, let  $G^{\even} - x_i$  be the graph derived from  $G^{\even}$ by deleting $x_i$ and its incident edge, and similarly let $G^{\odd} - x_i$  be the graph derived from  $G^{\odd}$ by deleting $x_i$ and its incident edge. Then $G^{\even} - x_i$  and $G^{\odd} - x_i$ are equivalent.
\end{lemma}
\begin{proof}

We will describe a set of isomorphisms between subgraphs of $G^{\even} - x_i$  and $G^{\odd} - x_i$, then combine them to produce an isomorphism between $G^{\even} - x_i$  and $G^{\odd} - x_i$. Each isomorphism will be one that maps vertex $z_{w,j}$ to $z_{\inv{w}{i},j}$.

For each $w \in S^{\even}$, Lemma~\ref{lem:caterpillarIsomorphism} implies that there exists an isomorphism $f$ between $Cat(w)$ and $Cat({\inv{w}{i}})$ such that $f(z_{w,j}) = z_{\inv{w}{i},j}$ for each $j \in [r]$.
For each $j \in [r]\setminus \{i\}$ and $h \in \{0,1\}$, Lemma~\ref{lem:specificLexTreeIsomorphism} implies that there exists an isomorphism $f$ between $Lex({j,h})^{\even}$ and $Lex({j,h})^{\odd}$, such that $f(v_{j,h}) = v_{j,h}$ and $f(z_{w,j},j) = z_{\inv{w}{i};j}$ for each leaf $z_{w,j}$. Finally, for each  $h \in \{0,1\}$, Lemma~\ref{lem:specificLexTreeIsomorphism} implies that there exists an isomorphism $f$ between $Lex({i,h})^{\even}$ and $Lex({i,1-h})^{\odd}$, such that $f(v_{i,h}^\even) = v_{i,1-h}^\odd$ and $f(z_{w,i}) = z_{\inv{w}{i},i}$ for each leaf $z_{w,i}$.

Observe that all of these isomorphisms agree on  $z_{w,j}$ for any $w \in S^\even, j \in [r]$ (that is, they each map this vertex to $z_{\inv{w}{i},j}$), and such vertices are the only vertices that are shared between caterpillars and lexicographic trees. Thus we can combine these isomorphisms into a single edge-preserving function $f$ that maps every non-leaf vertex of  $G^{\even} - x_i$ to a non-leaf vertex of  $G^{\odd} - x_i$.  Moreover, as each caterpillar and lexicographic tree in $G^{\even} - x_i$ is mapped to a different caterpillar or lexicographic tree in $G^{\odd} - x_i$, this function is a bijection.  Finally, set $f(x_j)=x_j$ for every $j \in [r]\setminus{i}$.  Then $f$ is now a bijective function from $V(G^{\even} - x_i)$ to $V(G^{\odd} - x_i)$ that is both edge-preserving and label-preserving.
\end{proof}

We note that we cannot extend the above graph isormorphism between $G^{\even} - x_i$  and $G^{\odd} - x_i$ to an isomorphism between $G^{\even}$  and $G^{\odd}$ by setting $f(x_i) = x_i$, because $f(v_{i,0}^\even) = v_{i,1}^\odd$, and so there would be no edge between $x_i=f(x_i)$ and $f(v_{i,0}^\even) = v_{i,1}^\odd$ in $G^{\odd}$ .

In fact, the next lemma shows that there is no isomorphism between $G^{\even}$  and $G^{\odd}$.

\medskip

\begin{lemma}\label{lem:GNonEquivalence}
Let ${\bf 0}$ denote the all-$0$ sequence from ${\cal B}_r$. For two vertices $a,b$ in $G^{\even}$, let $dist^\even(a,b)$ denote the distance between $a$ and $b$ in $G^\even$.  Similarly for two vertices $a,b$ in $G^{\odd}$, let $dist^\odd(a,b)$ denote the distance between $a$ and $b$ in $G^\odd$. Then for any vertex $a$ in $G^{\odd}$:
  \begin{enumerate}
\item If $d^\odd(a,x_1) = d^\even({\bf 0}, x_1)$ then $a = u_w$ for some $w \in S^{\odd}$.
\item If $a = u_w$ for some $w \in S^{\odd}$ then there exists $i \in [r]$ such that $dist^\odd(a,x_i) > dist^\even({\bf 0}, x_i)$.
  \end{enumerate}

This holds even if we suppress all degree-$2$ vertices in  $G^{\even}$  and $G^{\odd}$.
\end{lemma}
\begin{proof}
We consider the two parts of the claim separately.

\begin{enumerate}
\item  We first calculate the value of $d^\even({\bf 0}, x_1)$.
   Recall that in $G^{\even}$, $x_1$ is adjacent to the root $v_{1,0}$ of  $Lex({1,0})^{\even}$, and (by definition) every leaf of $Lex({1,0})^{\even}$ has distance $r-2$ from $v_{1,0}$.
   As $u_{\bf 0}$ is adjacent to a leaf $z_{{\bf 0},1}$ of  $Lex({1,0})^{\even}$, it follows that $d^\even({\bf 0}, x_1) = 1 + r-2 + 1 = r$ (there is no shorter path from $u_{\bf 0}$ to $x_1$, as any path must pass through $z_{w,1}$ for some $w$).

   As all leaves in $Lex({1,0})^{\odd}$ have distance $r-2$ from $v_{1,0}$ in $G^{\odd}$, and therefore distance $r-1 = d^\even({\bf 0}, x_1)-1$ from $x_1$, it follows that the only vertices in $G^{\odd}$ of distance $d^\even({\bf 0}, x_1)$ from $x_1$ are those which are not in $Lex({1,0})^{\odd}$ but adjacent to a leaf $z_{w,1}$ of $Lex({1,0})^{\odd}$. By construction, all such vertices are  $u_w$ for some $w \in S^{\odd}$ such that $\cha{w}{1}=0$.

   When degree-$2$ nodes are suppressed, a similar argument holds, except that $d^\even({\bf 0}, x_1)$ is reduced by $1$ (as we suppress $z_{{\bf 0},1}$). It remains the case that the vertices in $G^{\odd}$ of distance $d^\even({\bf 0}, x_1)$ from $x_1$ are those which are incident to a vertex from $Lex({1,0})^{\odd}$ but not in $Lex({1,0})^{\odd}$ themselves, and again  all such vertices are $u_w$ for some $w \in S^{\odd}$ .

\item For any $w \in S^{\odd}$, there exists $i \in [r]$ such that$\cha{w}{i}=1$. Any path from $u_w$ to $x_i$ must pass through a vertex $z_{w',i}$ where $\cha{w'}{i}=0$, and all such vertices have equal distance from $x_i$. Thus, it is enough to show that the distance in $G^{\odd}$ between $u_w$ and any such $z_{w',i}$ is greater than the distance between $u_{\bf 0}$ and $z_{{\bf 0},i}$ in $G^{\even}$.

    To see this, consider a path $P$ between $u_w$ and $z_{w',i}$.  As $\cha{w'}{i}=0$, we note that $w' \neq w$ and so $P$ must traverse at least one lexicographic tree. We construct a mapping $g:V(P) \rightarrow V(Cat({\bf 0}))$, as follows. For any $a \in V(P)$, if $a$ is in $Cat(w'')$ for any $w'' \in S^{\odd}$ (including $w$ or $w'$), set $g(a) = f(a)$, where $f$ is the isomorphism between $Cat({w''})$ and $Cat({\bf 0})$ such that $f(u_w) = u_{\bf 0}$ and $f(z_{w'',j}) = z_{{\bf 0},j}$ for all $j \in [r]$ (such an isomorphism exists by Observation~\ref{lem:caterpillarIsomorphism}). Otherwise, it must be the case that $a \in Lex(j,h)^{\odd}$ for some $j \in [r], h \in \{0,1\}$. In this case, set $g(a) = z_{{\bf 0},j}$.
    Let $Q$ be the set of all $g(a)$ for any node $a$ in $P$. Observe that for any nodes $a,b$ in $P$, if $a$ and $b$ are adjacent then either $g(a)=g(b)$ or $g(a)$ and $g(b)$ are adjacent.  It follows that $Q$ forms a connected set of nodes in $Cat({\bf 0})$, and thus $Q$ contains a path between $g(u_w) = u_{\bf 0}$ and $g(z_{w',i}) = z_{{\bf 0},i}$.
    Moreover, as $P$ must traverse at least one lexicographic tree, there are consecutive nodes in $P$ that are mapped to the same node by $g$. It follows that the path in $Q$ is shorter than the path $P$, as required.
    It follows that the distance between $u_w$ and $x_i$ is greater than $d^\even({\bf 0}, x_i)$. We note that a similar argument applies even when vertices of degree $2$ are suppressed.
  \end{enumerate}
 \end{proof}

\medskip

\begin{corollary}
$G^{\even}$  and $G^{\odd}$ are not equivalent.
\end{corollary}

The next lemma will be used to show that when we suppress degree-$2$ vertices in $G^{\even}$ and $G^{\odd}$, the resulting graphs $N^{\even}$ and $N^{\odd}$ are networks.

\begin{lemma}\label{lem:nonLeavesShareCycles}
In both $G^{\even}$ and $G^{\odd}$, there exists a single blob containing all non-leaf vertices.
\end{lemma}

\begin{proof}
Observe that any non-leaf vertex is part of a path between $u_w$ and $v_{i,h}$ for some $w \in {\cal B}_r$, $i \in [r], h \in \{0,1\}$. Furthermore every vertex $v_{i,h}$ appears on a path between $u_w$ and $u_{w'}$ for some $w,w'$.   Therefore it is enough to show that for any $w \neq w'$, $u_w$ and $u_w'$ appear in the same blob.

Let $00*,01*,11*,10*$ be four sequences in $S^{\even}$ such that $\cha{hk*}{1}=h, \cha{hk*}{2} = k$ (such sequences exist as $r \geq 3$).

Then there exists a cycle
$$
u_{00*}z_{00*,1} \dots  z_{01*,1}u_{01*}z_{01*,2} \dots z_{11*,2}u_{11*}z_{11*,1} \dots z_{10*,1}u_{10*}z_{10*,2} \dots z_{00*,2}u_{00*}.
$$
Here the path between $z_{00*,1}$ and  $z_{01*,1}$ passes through $Lex({1,0})^\even$, the path between $z_{01*,2}$ and  $z_{11*,2}$ passes through $Lex({2,1})^\even$, the path between $z_{11*,1}$ and  $z_{10*,1}$ passes through $Lex({1,1})^\even$, and   the path between $z_{10*,2}$ and  $z_{00*,2}$ passes through $Lex({2,0})^\even$.
See Figure~\ref{fig:exampleCycle} for an example when $00* = 0000, 01* = 0101, 11* = 1100$ and $10* = 1001$.)

As $00*,01*,11*,10*$ appear on a cycle, they appear in the same blob of $G^{\even}$. Moreover as any node $u_w$ could fill the role of one of $00*,01*,11*,10*$, we have that all $u_w$ appear in the same blob. A similar argument holds for $G^{\odd}$.
\end{proof}

\begin{figure}
\begin{center}
\includegraphics[scale = .9]{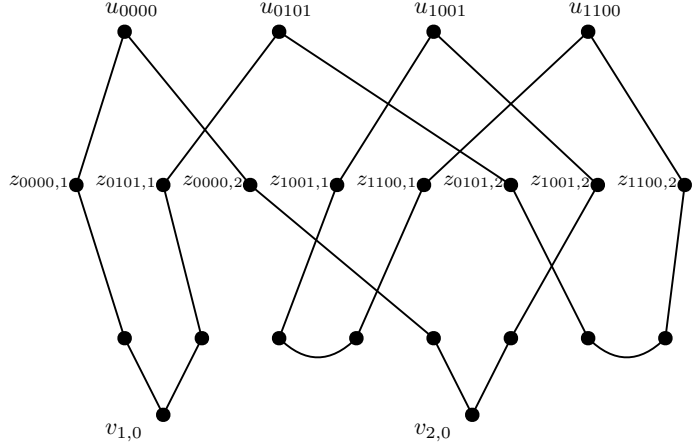}
\end{center}
\label{fig:exampleCycle}
\caption{A cycle containing the vertices $u_{0000}, u_{u_0101}, u_{1001}, u_{1100}$ in  $G^{\even}$ for the case $r = 4$.}
\end{figure}

Now we are ready to construct the networks  $N^{\even}$ and $N^{\odd}$:
Let  $N^{\even}$ be derived from $G^{\even}$ by suppressing all nodes of degree $2$. Similarly, let  $N^{\odd}$ be derived from $G^{\odd}$ by suppressing all nodes of degree $2$.

\medskip

\begin{lemma}\label{lem:NNetwork}
$N^{\even}$ and $N^{\odd}$ are networks on $X$.
\end{lemma}
\begin{proof}
We show that  $N^{\even}$ is a network on $X$ (the proof for  $N^{\odd}$ is similar). By construction, all vertices in  $N^{\even}$ have degree $1$ or $3$ and the leaves are bijectively labelled  with the elements of $X$.
It remains to show  that contracting each blob into a single vertex gives a tree with no degree-$2$ vertices, which we will do by showing that  $N^{\even}$ has only one blob.
By Lemma~\ref{lem:nonLeavesShareCycles}, all non-leaf vertices in  $G^{\even}$ are part of the same blob in  $G^{\even}$
Observe that if two degree-$3$ vertices are in the same blob, then they are still in the same blob after contracting degree-$2$ vertices.
Thus, all non-leaf vertices in  $N^{\even}$ are part of the same blob, and thus  $N^{\even}$  has a single blob, as required.
\end{proof}

  \medskip

\begin{lemma}\label{lem:NNonEquivalence}
$N^{\even}$ and $N^{\odd}$ are not equivalent.
\end{lemma}
\begin{proof}
As  $N^{\even}$ and $N^{\odd}$ are derived from  $G^{\even}$ and $G^{\odd}$ by suppressing degree-$2$ vertices, Lemma~\ref{lem:GNonEquivalence} implies that there is no vertex in  $N^{\odd}$ that has the same distance from each leaf $x_i$ as  $u_{\bf 0}$ has from $x_i$ in $N^{\even}$.

This implies that there is no isomorphism between $N^{\even}$ and $N^{\odd}$, as if $f$ is edge-preserving and label-preserving then the distance between $u_{\bf 0}$ and $x_i$ is equal to the distance between $f(u_{\bf 0})$ and $f(x_i)=x_i$.
\end{proof}

 \medskip

\begin{lemma}\label{lem:NiEquivalence}
For each $i \in [r]$, $(N^{\even})_{x_i}$ and $(N^{\odd})_{x_i}$ are equivalent.
\end{lemma}
\begin{proof}
Recall the definitions of $G^{\even} - x_i$ and $G^{\odd} - x_i$, and observe that $(N^{\even})_{x_i}$ (respectively, $(N^{\odd})_{x_i}$) can be derived from $G^{\even} - x_i$  ($G^{\odd} - x_i$) by suppressing degree-$2$ vertices. By Lemma~\ref{lem:GiEquivalence}, there exists an isomorphism $f'$ between $G^{\even} - x_i$ and $G^{\odd} - x_i$. So define a bijective function $f: V((N^{\even})_{x_i}) \rightarrow V((N^{\odd})_{x_i})$   by setting $f(a) = f'(a)$ for all $u \in V((N^{\even})_{x_i})$. Note that if $a$ does not have degree $2$ in $G^{\even} - x_i$, $f'(a)$ also does not have degree $2$ in $G^{\odd} - x_i$.   Thus if $a \in  V((N^{\even})_{x_i})$ then $f(a) = f'(a) \in V((N^{\odd})_{x_i})$, and so $f$ is well-defined.

By construction, $f$ is label-preserving. To see that $f$ is edge-preserving, consider some $a,b \in V((N^{\even})_{x_i})$. Observe that the number of edges between $a$ and $b$ in $(N^{\even})_{x_i}$ is equal to the number of paths between $a$ and $b$ in $G^{\even} - x_i$ whose internal vertices have degree $2$. As $f'$ is an isomorphism, this is equal to the number of paths between $f'(a)$ and $f'(b)$ in $G^{\odd} - x_i$ whose internal vertices have degree $2$, which in turn is equal to the number of edges between $f(a)$ and $f(b)$ in  $(N^{\odd})_{x_i}$. Thus, $f$ is edge-preserving, and so $f$ is an isomorphism.
\end{proof}

Lemmas~\ref{lem:NNetwork},~\ref{lem:NNonEquivalence} and ~\ref{lem:NiEquivalence} imply the following theorem:

 \medskip

\begin{theorem}
For any $r \geq 4$, there exist networks $N^{\even}$, $N^{\odd}$ on $X$ with $|X| = r$, such that $N^{\odd}$ is a leaf-reconstruction of $N^{\even}$, but $N^{\even}$ and $N^{\odd}$ are not equivalent. Thus, $N^{\even}$ is not leaf-reconstructible.
\end{theorem}

\begin{figure}[H]
\begin{subfigure}{\textwidth}
  \tikzstyle{black}=[shape=circle, fill=black, radius=0.09, inner sep = 2.3]
\centering\begin{tikzpicture}[ very thick, scale = 0.8]
\foreach \eps in {0.45}
{
\foreach \xi in {1,2,3,4}
{
\node at (4*\xi-3.5,1)(root\xi0)[black]{};
\node at (4*\xi-3.7-1*\eps,1-0.7*\eps){$v_{\xi,0}^\even$};
\node at (4*\xi-3.5,0)[black]{}
 edge (root\xi0);
\node at (4*\xi-3.5,0-\eps){$x_{\xi}$};
\node at (4*\xi-4,1.5)(root\xi070)[black]{}
 edge (root\xi0);
\node at (4*\xi-3,1.5)(root\xi071)[black]{}
 edge (root\xi0);
\node at (4*\xi-2.3,1.5)(root\xi170)[black]{};
\node at (4*\xi-1.5,1.5)(root\xi171)[black]{}
 edge[bend left=20] (root\xi170);
}
\foreach \ha in {0,1}
\foreach \hb in {0,1}
\foreach \hc in {0,1}
\foreach \hd in {0,1}
{
\ifthenelse{\intcalcMod{\ha + \hb + \hc + \hd}{2} = 0}{
\node at (8*\ha + 4*\hb + 2*\hc,6)(root)[black]{}
edge (root1\ha7\hb)
edge (root2\hb7\ha)
;
\node at (8*\ha + 4*\hb + 2*\hc + 0.75,5.75)(root1)[black]{}
edge (root)
edge (root3\hc7\ha)
edge (root4\hd7\ha)
;
\node at (8*\ha + 4*\hb + 2*\hc,6+\eps)(root){$u_{\ha \hb \hc \hd}$};
}{}
}
}
\end{tikzpicture}
\caption{$N^{\even}$}\label{fig:NEvenBinary}
\end{subfigure}

\begin{subfigure}{\textwidth}
\label{fig:hardnessGadgetForX}
  \tikzstyle{black}=[circle, fill=black, radius=0.09, inner sep = 2.3]
\centering\begin{tikzpicture}[ very thick, scale = 0.8]
\foreach \eps in {0.45}
{
\foreach \xi in {1,2,3,4}
{
\node at (4*\xi-3.5,1)(root\xi0)[black]{};
\node at (4*\xi-3.7-1*\eps,1-0.7*\eps){$v_{\xi,0}^\odd$};
\node at (4*\xi-3.5,1-1)[black]{}
 edge (root\xi0);
\node at (4*\xi-3.5,1-1-\eps){$x_{\xi}$};

\node at (4*\xi-4,1.5)(root\xi070)[black]{}
 edge (root\xi0);
\node at (4*\xi-3,1.5)(root\xi071)[black]{}
 edge (root\xi0);
\node at (4*\xi-2.3,1.5)(root\xi170)[black]{};
\node at (4*\xi-1.5,1.5)(root\xi171)[black]{}
 edge[bend left=20] (root\xi170);
}
\foreach \ha in {0,1}
\foreach \hb in {0,1}
\foreach \hc in {0,1}
\foreach \hd in {0,1}
{
\ifthenelse{\intcalcMod{\ha + \hb + \hc + \hd}{2} = 1}{
\node at (8*\ha + 4*\hb + 2*\hc,6)(root)[black]{}
edge (root1\ha7\hb)
edge (root2\hb7\ha)
;
\node at (8*\ha + 4*\hb + 2*\hc + 0.75,5.75)(root1)[black]{}
edge (root)
edge (root3\hc7\ha)
edge (root4\hd7\ha)
;
\node at (8*\ha + 4*\hb + 2*\hc,6+\eps)(root){$u_{\ha \hb \hc \hd}$};
}{}
}
}
\end{tikzpicture}
\caption{$N^{\odd}$}\label{fig:NOddBinary}
\end{subfigure}
\caption{ Binary example of  $N^{\even}$ and $N^{\odd}$  for the case when when $r=4$.  The node $u_{0000}$ in $N^{\even}$ has distance $d_1 = 3$ from $x_1$, $d_2 = 3$ from $x_2$, $d_3= 4$ from $x_3$, and $d_4 = 4$ from $x_4$.
Moreover there is no node in $N^{\odd}$  with distance $d_i$ from $x_i$ for each $i \in [4]$.  Thus $N^{\even}$ and $N^{\odd}$ are not equivalent.
However, for each $i \in [4]$ the multigraphs $(N^{\even})_{x_i}$ and $(N^{\odd})_{x_i}$ are equivalent, using an isomorphism that maps each vertex $u_w$ to $u_{\inv{w}{i}}$.
}
\end{figure}
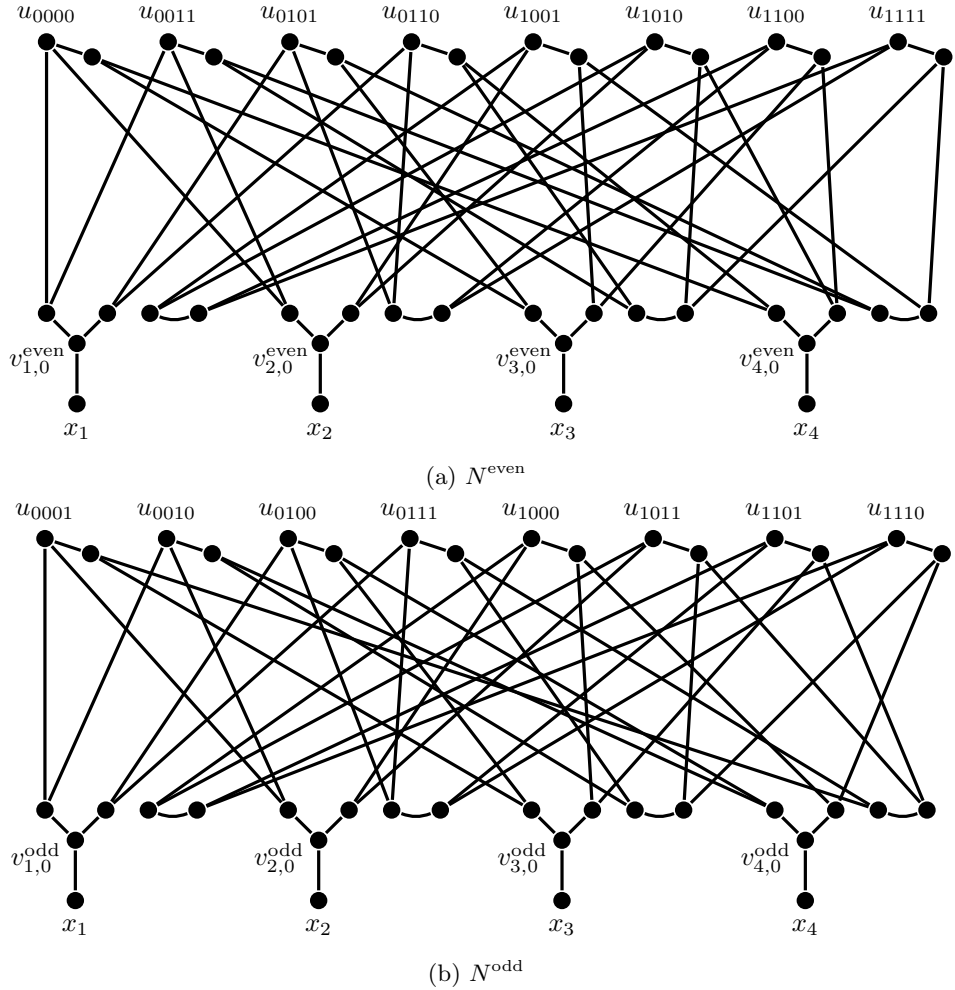

\section{Concluding remarks}\label{sec:conclusion}
Although we have shown that not all phylogenetic networks with five or
more leaves are leaf-reconstructible, this does not mean that
reconstructing networks from subnetworks is completely hopeless. There
are already some positive results for interesting restricted network
classes~\cite{vIM2017}. Moreover, since the presented counter examples
are very complex, it is certainly possible that other reasonable
network classes are also leaf-reconstructible.

For example, while it is known that all networks with at least five
leaves and~$|E|-|V|\leq 3$ are leaf-reconstructible, the counter
examples presented in this paper have $|E|-|V|=(2^{r-1}-1)(2r-1)-1$,
with~$r$ the number of leaves. Hence, whether networks with
$3<|E|-|V|<(2^{r-1}-1)(2r-1)-1$ are leaf-reconstructible is still
open. In particular, is it possible to construct counter examples
where $|E|-|V|$ is bounded by a linear function of the number of
leaves?

\end{document}